\documentclass[11pt]{amsart}

\usepackage[top=1.5in, bottom=1in, left=0.9in, right=0.9in]{geometry}

\usepackage{amscd,amsmath,amssymb,fancyhdr,color}
\usepackage[utf8]{inputenc}
\usepackage{amsfonts}
\usepackage{amsthm}
\usepackage{accents}
\usepackage{graphicx}
\usepackage{float}
\usepackage{subcaption}
\usepackage{verbatim}
\usepackage{indentfirst}
\usepackage{dsfont}
\usepackage{tikz}
\usepackage{tikz-cd}
\usetikzlibrary{matrix}
\usepackage[all]{xy}
\usepackage{enumerate}
\usepackage{extarrows}
\usepackage{csquotes}

\usepackage{scalerel,stackengine}
\stackMath
\newcommand\reallywidehat[1]{%
	\savestack{\tmpbox}{\stretchto{%
			\scaleto{%
				\scalerel*[\widthof{\ensuremath{#1}}]{\kern-.6pt\bigwedge\kern-.6pt}%
				{\rule[-\textheight/2]{1ex}{\textheight}}
			}{\textheight}%
		}{0.5ex}}%
	\stackon[1pt]{#1}{\tmpbox}%
}

\usepackage{faktor}

\usepackage{BOONDOX-uprscr}

\usepackage[backref=page]{hyperref}
\renewcommand*{\backref}[1]{}
\renewcommand*{\backrefalt}[4]{%
	\ifcase #1 (Not cited.)%
	\or        (Cited on page~#2.)%
	\else      (Cited on pages~#2.)%
	\fi}

\hypersetup{
	colorlinks   = true,
	citecolor    = magenta
}

\DeclareFontFamily{U}{MnSymbolC}{}
\DeclareSymbolFont{MnSyC}{U}{MnSymbolC}{m}{n}
\DeclareFontShape{U}{MnSymbolC}{m}{n}{
    <-6>  MnSymbolC5
   <6-7>  MnSymbolC6
   <7-8>  MnSymbolC7
   <8-9>  MnSymbolC8
   <9-10> MnSymbolC9
  <10-12> MnSymbolC10
  <12->   MnSymbolC12}{}
\DeclareMathSymbol{\intprod}{\mathbin}{MnSyC}{'270}

\newcommand{\K}{K\"ahler}

\DeclareMathSymbol{\cs} {\mathord}{AMSb}{"60}
\DeclareMathSymbol{\intprod}{\mathbin}{MnSyC}{'270}

\numberwithin{equation}{section}

\def\eqref#1{(\ref{#1})}

\newcommand{\C}{{\mathbb C}}
\newcommand{\R}{{\mathbb R}}

\newcommand{\6}{\partial}

\newcommand{\del}{\partial}

\def\1{\sqrt{-1}\:}
\newcommand{\restrict}[1]{{\big|_{{\phantom{|}\!\!}_{#1}}}}
\newcommand{\cntrct}                
{\hspace{2pt}\raisebox{1pt}{\text{$\lrcorner$}}\hspace{2pt}}

\renewcommand{\dim}{\operatorname{dim}}

\newcommand{\Mrond}{\mathcal{M}}

\newcommand{\Urond}{\mathcal{U}}

\newcounter{Mycounter}[section]
\newcounter{lemma}[section]
\newcounter{claim}[section]
\newcounter{sublemma}[section]
\newcounter{corollary}[section]
\newcounter{theorem}[section]
\newcounter{conjecture}[section]
\newcounter{proposition}[section]
\newcounter{definition}[section]
\newcounter{example}[section]
\newcounter{remark}[section]
\newcounter{problem}[section]
\newcounter{question}[section]
\makeatletter

\@addtoreset{equation}{section}

\@addtoreset{footnote}{section}

\makeatother

\makeatletter
\DeclareRobustCommand*{\mfaktor}[3][]
{
   { \mathpalette{\mfaktor@impl@}{{#1}{#2}{#3}} }
}
\newcommand*{\mfaktor@impl@}[2]{\mfaktor@impl#1#2}
\newcommand*{\mfaktor@impl}[4]{
   \settoheight{\faktor@zaehlerhoehe}{\ensuremath{#1#2{#3}}}%
   \settoheight{\faktor@nennerhoehe}{\ensuremath{#1#2{#4}}}%
      \raisebox{-0.5\faktor@zaehlerhoehe}{\ensuremath{#1#2{#3}}}%
      \mkern-4mu\diagdown\mkern-5mu%
      \raisebox{0.5\faktor@nennerhoehe}{\ensuremath{#1#2{#4}}}%
}
\makeatother

\usetikzlibrary{arrows,chains,matrix,positioning,scopes}

\makeatletter
\tikzset{join/.code=\tikzset{after node path={%
			\ifx\tikzchainprevious\pgfutil@empty\else(\tikzchainprevious)%
			edge[every join]#1(\tikzchaincurrent)\fi}}}
\makeatother

\tikzset{>=stealth',every on chain/.append style={join},
	every join/.style={->}}

\makeatletter
\newtheorem*{rep@theorem}{\rep@title}
\newcommand{\newreptheorem}[2]{%
	\newenvironment{rep#1}[1]{%
		\def\rep@title{\ref{##1}}%
		\begin{rep@theorem}}%
		{\end{rep@theorem}}}
\makeatother

\newreptheorem{theorem}{Theorem}

\newtheoremstyle{named}{}{}{\itshape}{}{\bfseries}{.}{.5em}{\thmnote{#3's }#1}
\theoremstyle{named}

\begin{document}
	
	\newpage
	
	\title[Special Hermitian metrics]{Special Hermitian metrics}

        \author{Cristian Ciulic\u a}
	\address{Cristian Ciulic\u a  \newline
		\textsc{\indent University of Bucharest, Faculty of Mathematics and Computer Science\newline 
			\indent 14 Academiei Str., Bucharest, Romania}}
	\email{cristiciulica@yahoo.com}
	
	\date{\today
	}

	\begin{abstract}
		We study the stability at blow-up and deformations of a class of Hermitian metrics whose fundamental two-form $\omega$ satisfies the condition $\6 \bar \6 \omega^k=0$, for any $k$ between 1 and $n-1$ (where $n$ is the complex dimension of the manifold). We are motivated by the existence of compact complex manifold supporting such metrics.
	\end{abstract}
	
	\maketitle
	
	\hypersetup{linkcolor=blue}

	\section{Introduction}
	Among all complex manifolds, those admitting Kähler structures occupy a privileged place: they lie at the intersection of Riemannian, symplectic, and complex geometry, and they often lend themselves to algebraic techniques, as they are frequently projective in the compact case. On the other hand, many interesting compact complex manifolds do not admit K\"ahler metrics. Therefore,  various conditions have been introduced that relax Kählerianity while still maintaining enough control to derive meaningful results. Some prominent examples of non-K\"ahler metrics are pluriclosed  (or SKT), balanced, locally conformally Kähler, and astheno-Kähler. More details about each of these types of metrics, as well as the motivation for their introduction, can be found in \cite{b87}, \cite{m82}, \cite{ov24}, \cite{jy}, \cite{fww13}.

    The question of when does a complex manifold admit metrics in several of the aforementioned classes has been intensively studied. In this note, we focus on metrics with Hermitian $2$- form $\omega$ satisfying $\6 \bar \6 \omega^k=0$, for any $k$ between 1 and $n-1$ (where $n=\dim_\C X$). We call this metrics {\em special} (by lack of a more suggestive name). Such a metric is, in particular,  both pluriclosed ($k=1$) and astheno-Kähler ($k=n-2$). 
    
    Our motivation comes from the study of Endo-Pajitnov manifolds (\cite{pajitnov1}), which are non-K\"ahler compact manifolds, generalizations in higher dimensions of the Inoue surfaces of type $S^+$. In  the recent paper \cite{com}, we explicitly wrote a Hermitian metric which is special in the above sense on an Endo-Pajitnov manifold.  
    
    On the other hand, in  \cite{fgv19} compact complex manifolds carrying both astheno-K\"ahler and SKT metrics are constructed. Note that, in general, metrics which belong simultaneously to two different classes of non-K\"ahler metrics do not exist: for example, a metric which is both astheno-K\"ahler and balanced, or both SKT and balanced should be K\"ahler (\cite{fv}, \cite{_Alexandrov_Ivanov_}).

    In this note, we discuss the stability of special Hermitian metrics at deformations and with respect to blow-up along submanifolds.
    
   Concerning the stability at deformations, we follow and generalize the work of  T. Sferruza  for astheno-Kähler, SKT and balanced metrics (\cite{sfe}, \cite{sfe_2}, \cite{ps}).

    As for the blow-up, we investigated Hermitian metrics satisfying a more general property than being special: for any fixed $k$, we ask $\6 \bar \6 \omega^i=0$ for any $i=\overline{1,k}$. It turns out that this condition is stable with respect to blow-up along submanifolds. Our result then generalizes the ones about the blow-up of SKT  (\cite{ft}) metrics, and complements the ones concerning the blow-up of other non-K\"hler types of metrics (e.g. \cite{vul}, \cite{ovv13} for LCK).
    
    Note that this result contrasts the one in \cite{sfto} where it is proven that a metric satisfying  $\6 \bar \6 \omega^i=0$ for $i=n-2$ {\em and} $i=n-3$ is not preserved by blow-up. 

    The paper is organized as follows. In  Section \ref{not}, we fix the notations that will be used throughout the paper. The next two sections contain the main results; each of them begins with a subsection containing introductory material.

        \section{Conventions and notations}
        \label{not}

        Let $M$ be a differentiable manifold of dimension $2n$ with a structure of complex manifold. A structure of complex manifold is equivalent to an integrable almost complex structure $J$ on $M$, i.e. an endomorphism of the tangent bundle $TM$ such that $J^2=-Id_{TM}$ and the Nijenhuis tensor $N_J$ identically vanishes, that is
        \begin{equation*}
            N_J(X,Y)=[X,Y]-J[JX,Y]-J[X,JY]-[JX,JY]=0
        \end{equation*}
        for every $X,Y\in TM$.
        We will denote the extension of $J$ to the complexification of tangent bundle $TM_\mathbb{C}$ by $J$ as well; the eigenvalues of this operator are $\pm i$. Hence, we will obtain a direct sum decomposition of $TM_\mathbb{C}$, $TM_\mathbb{C}=T^{1,0}M\oplus T^{0,1}M$ in terms of the $\pm i$-eigenspaces of $J$.

        The complex structure acts also on differential forms of degree $k$ by:  $J\alpha=i^{q-p} \alpha$, for any $\alpha \in \Lambda^{p, q}_{\mathbb{C}}M$, or equivalently $$J\eta(X_1, \ldots, X_p) = (-1)^p\eta(JX_1, \ldots, JX_p).$$

        A Riemannian metric $g$ on $(M,J)$ is called Hermitian if $g(JX,JY)=g(X,Y)$, for every $X,Y \in TM$. The fundamental $2-$form $\omega$ of $g$ is defined by $\omega(X,Y):=g(JX,Y)$, for every $X,Y \in TM$.

        We denote by $A^k_\C (M)$ the space of complex-valued $k$-forms on $M$ and by $A^{p,q}(M)$ the space of $(p,q)$-forms with respect to the complex structure $J$.

        We recall the action of the differential operator:
        \[
        d: A^{p,q}(M) \rightarrow A^{p+1, q}(M) \oplus A^{p, q+1}(M)
        \]
        which decomposes as $d=\6 + \bar\6$, where $\6:= \pi^{p+1,q} \circ d$, $\bar\6 := \pi^{p,q+1} \circ d$.

        Let $\pi: E \rightarrow M$ be a rank $r$ holomorphic vector bundle and $A^{p,q}(M,E) := A^{p,q}(M) \otimes_{\mathcal{C}^\infty(M)} E$.

        We define $\bar\6_E$ on $A^{p,q}(M,E)$ as follows: let $\varphi := \sum\limits_i \eta^i \otimes s_i \in A^{p,q}(M,E)$, where $\eta^i \in A^{p,q}(M)$ and $(s_1, \ldots, s_r)$ is the expression of a section $s \in \Gamma(E)$ in a local trivialisation; then we set:
        \[
            \bar\6_E \varphi := \sum\limits_i \bar\6 (\eta^i) \otimes s_i
        \]

        \begin{definition}
        Let $\xi = \alpha \otimes s \in A^{p,q}(E)$ and $\varphi = \bar \eta \otimes z \in A^{0,r}(T^{1,0} M)$. Following \cite{rz} and \cite{sfe}, we define the contraction by $\varphi$ as:
        \begin{align*}
            \iota_\varphi: A^{p,q}(E) &\rightarrow A^{p-1, q+1}(E) \\
            \xi &\mapsto \iota_\varphi(\xi) := (\bar \eta \wedge \iota_Z(\alpha)) \otimes s
        \end{align*}
        \end{definition}

\section{Deformations of special Hermitian metrics}
        
        \subsection{Review of deformation theory}
        \label{sec:prelim}

        In this section, we will recall some definitions and properties related to the deformation of complex structures. We will review the formulas for differential operators on each element of a differentiable family of deformations, and these formulas will be expressed in terms of the operators on $(M,J)$. 
        
        \begin{definition} (\cite[Definition 4.1]{kod_2} \cite[Definition 2.1]{rz})
        \label{def:complex_deformation}
            Let $X$ be a smooth manifold, $B \subset \R^k$ a domain and $\pi: X \rightarrow B$ a differentiable surjection.

            The triple $(\pi, X, B)$ is called a differentiable family of compact complex manifolds of dimension $n$ if:
            \begin{enumerate}
                \item the rank of the Jacobian of $\pi$ is equal to $k$ at every point of $X$,
                \item for any $t \in B$, $X_t:=\pi^{-1}(t)$ is a compact connected submanifold, endowed with a complex structure,
                \item there exists a locally finite cover $\{ \mathcal{U}_j \rvert  j \in I \}$ of $X$ and differentiable complex-valued functions $\xi^1_j, \ldots, \xi_j^n$ defined on $\mathcal{U}_j$ such that for any $t \in B$, the atlas 
                \[ 
                \{ \left(\mathcal{U}_j \cap \pi^{-1}(t), p \mapsto (\xi_j^1(p),\ldots,\xi_j^n(p))\right)\ \rvert\ j \in I, \mathcal{U}_j \cap \pi^{-1}(t) \neq \emptyset\}
                \]
                forms a system of holomorphic coordinates on $X_t$.
            \end{enumerate}
        \end{definition}

        We also recall from \cite{rz} the following definition of the Lie bracket of $T^{1,0}M$-valued forms.

    \begin{definition}
        Let $\varphi \in A^{0,p}(T^{1,0}M)$, $\Psi \in A^{0,q}(T^{1,0}M)$. Locally, we can write:
        \begin{align*}
            \varphi &= \sum_{\bar{j_1}, \ldots,\bar{j_p}, i}\varphi^i_{\bar{j_1}, \ldots, \bar{j_p}}d\bar{z}^{j_1} \wedge \cdots \wedge d\bar{z}^{j_p} \otimes \frac{\6}{\6z_i} \\
            \Psi &= \sum_{\bar{k_1}, \ldots, \bar{k_q}, i}
                \Psi^i_{\bar{k_1}, \ldots, \bar{k_q}} d\bar{z}^{k_1} \wedge \cdots \wedge d\bar{z}^{k_q} \otimes \frac{\6}{\6z_i}
        \end{align*}
        We define:
        \[
        [\varphi, \Psi] := \sum_{i,j=1}^n \left(\varphi^i \wedge \6_i \Psi^j - (-1)^{pq} \Psi^i \wedge \6_i \varphi^j\right) \otimes \frac{\6}{\6z_j} \quad, 
        \]
        where:
        \begin{align*}
        \varphi^j := \sum_{\bar{j_1}, \ldots,\bar{j_p}} \varphi^j_{\bar{j_1}, \ldots, \bar{j_p}}d\bar{z}^{j_1} \wedge \cdots \wedge d\bar{z}^{j_p} \\
        \6_i \varphi^j := \sum_{\bar{j_1}, \ldots, \bar{j_p}} \frac{\6}{\6z_i}\left(\varphi^j_{\bar{j}_1, \ldots, \bar{j}_p} \right) d\bar{z}^{j_1} \wedge \cdots \wedge d\bar{z}^{j_p}
        \end{align*}
        and $\Psi^j, \6_i \Psi^j$ are defined analogously.
        \end{definition}

        In particular, if $\varphi, \Psi \in A^{0,1}(T^{1,0}M)$, then:
        \[
        [\varphi, \Psi] = \sum_{i,j=1}^n
        \left(\varphi^i \wedge \6_i \Psi^j + \Psi^i \wedge \6_i \varphi^j\right) \otimes \frac{\6}{\6 z_j}
        \]

        \begin{definition}
            For $\phi \in A^{0,q}(T^{1,0}M)$. We define the twisted Lie derivative:
            \[
            \mathcal{L}_\phi := (-1)^q d \circ \iota_\phi + \iota_\phi \circ d
            \]
        \end{definition}

        \begin{remark}
        We can decompose $\mathcal{L}_\phi$ as:
        \[
            \mathcal{L}_\phi := \mathcal{L}_\phi^{1,0}+  \mathcal{L}_\phi^{0,1} \quad , 
        \]
        where 
        \begin{align*}
             \mathcal{L}_\phi^{1,0} &= (-1)^q \6 \circ \iota_\phi + \iota_\phi \circ \6 \\
              \mathcal{L}_\phi^{0,1} &= (-1)^q \bar\6 \circ \iota_\phi + \iota_\phi \circ \bar\6
        \end{align*}
        \end{remark}
        \vspace{5 mm}

        The following definition is of central importance.

\begin{definition}
            Let $\phi \in A^{0,1}(T^{1,0}M)$. We define  $e^{\iota_ \phi}:A^*(M) \longrightarrow A^*(M)$ by: 
            \[
                e^{\iota_ \phi} =\sum_{k=0}^\infty \frac{1}{k!} \iota_\phi^k.
            \]
\end{definition}
\begin{proposition}{\rm (\cite[Theorem 3.4]{lry})} 
            Let $\phi \in A^{0,1}(T^{1,0}M)$. We have:
            \[
            e^{-\iota_ \phi} \circ d \circ e^{\iota_\phi}=d - \mathcal{L}_\phi -\iota_{\frac{1}{2}[\phi, \phi]}=d-\mathcal{L}_\phi^{1,0} + \iota_{\bar\6 \phi - \frac{1}{2} [\phi, \phi]}
            \] i.e.:
            \begin{align*}
                e^{\iota_ \phi} \circ \bar\6 \circ e^{\iota_ \phi}
                &= \bar\6 - \mathcal{L}_\phi^{0,1} \\
                e^{-\iota_ \phi} \circ \6 \circ e^{\iota_ \phi}
                &= \6 - \mathcal{L}_\phi^{1,0}
                - \iota_{\frac{1}{2}[\phi,\phi]}
            \end{align*}
        \end{proposition}

        \vspace{5 mm}

        In what follows, we consider $\pi: \Mrond \rightarrow B_r$ a differentiable family of compact complex manifolds of complex dimension $n$, where $B_r$ is the unit ball in $\R^k$. More detailes about this description can be found in \cite{kod}.

        For simplification, we assume from now on that $k=1$. We denote by $M_0 := \pi^{-1}(0)$, $M_t := \pi^{-1}(t)$.

        We fix an open cover $\{ \Urond_j \}_{j \in I}$ such that $\Urond_j =
        \{ 
        \left(\xi_j, t 
        \right) = ( \xi_j^1, \ldots, \xi_j^n, t )\rvert \left\lvert \left \lvert \xi_j \right\lvert \right\rvert < 1, |t| < r
        \}$
        and ${\pi(\xi_j, t) = t}$ and ${\xi_j^\alpha = f_{j,k}^\alpha (\xi_k, t)}$ on $\Urond_j \cap \Urond_k$ where $f_{j,k}$ is holomorphic in $\xi_k$ and differentiable in $t$.

        Then by a classical result of Ehresmann (\cite {kod}), $\Mrond \stackrel{\textrm{diff}}{\simeq} M \times B_r$, where $M$ is the differentiable manifold underlying the complex manifold $M_0$. Via this diffeomorphism, we have $\Urond_j \stackrel{\textrm{diff}}{\simeq} U_j \times B_r$, where ${U_j = \{ \xi_j \rvert \left\lvert \left \lvert \xi_j \right\lvert \right\rvert < 1 \}}$. 
        
        For $x \in M$, $\xi_j^\alpha = \xi_j^\alpha (x, t)$ is a differentiable function of $(x,t)$. If $z$ is a holomorphic coordinate on $M_0 = M$, then for $t=0$, $\xi_j^\alpha(z,0)$ is holomorphic in $z$, while for $t \neq 0$, $\xi_j^\alpha(z,t)$ is merely differentiable, since the holomorphic coordinate $z$ on $M_0$ is no longer necessarily holomorphic on $M_t$ when $t \neq 0$.

        Starting from a compact complex manifold $(M,J) $ it is known (\cite{kod}) that we can construct families of deformations $(M_t)_{t \in B}$ of $(M,J)$, where $M_t=(M, J_t)$, for every $t \in B$. Here $J_t$ is is an integrable complex structure on the differentiable manifold $M$ parameterized by a $(0,1)-$vector form $\varphi(t)$ on $(M,J)$. We proceed to explain how to characterise the complex structures on each $M_t$. First, we remark: 

        \begin{remark}
        	We have the following formulae:
         \begin{enumerate}
             \item 
             \begin{align*}
                 \bar\6 \xi_j^\alpha(z,t) &=
                 \bar\6 \left( f_{j,k}^\alpha \left(\xi_k(z,t), t
                     \right) 
                 \right) = 
                 \sum_\beta \frac{\6}{\6 \bar\6z^\beta}
                 \left(f_{j,k}^\alpha 
                    \left(\xi_k(z,t), t
                    \right)
                 \right) d \bar z ^ \beta \\
                 &=
                 \sum_\beta
                 \left(\sum_\gamma
                    \left(\frac{\6 f_{j,k}^\alpha}{\6 \xi_k^\gamma}\cdot
                        \frac{\6\xi_k^\gamma}{\6 \bar z^\beta}+\frac{\6 f_{j,k}^\alpha}{\6 \bar\xi_k^\gamma}\cdot
                        \frac{\6\bar\xi_k^\gamma}{\6 \bar z^\beta}
                    \right)
                 \right) d \bar z ^ \beta \\
                 &=\sum_\gamma\left(\frac{\6 f_{j,k}^\alpha}{\6 \xi_k^\gamma}\left( \sum_\beta \frac{\6\xi_k^\gamma}{\6 \bar z^\beta}d \bar z ^ \beta
                        \right)
                 \right)
                 =\sum_\gamma\frac{\6 f_{j,k}^\alpha}{\6 \xi_k^\gamma}\bar\6 \xi_k^\gamma(z,t) \\
                &= 
                \sum_\beta \frac{\6 f_{j,k}^\alpha}{\6 \xi_k^\beta}\bar\6 \xi_k^\beta
             \end{align*}
        \item 
        \begin{align*}
            \frac{\6}{\6 z_\lambda} \left( \xi_j^\alpha(z,t)\right)
            =
            \sum_\beta
                \frac{\6 f_{j,k}^\alpha}{\6 \xi_k^\beta}
                \frac{\6}{\6 z_\lambda} \left(
                    \xi_k^\beta(z,t)
                \right)
        \end{align*}
         \end{enumerate}
        \end{remark}

        On $M_0$, $(z^1, \ldots, z^n)$ are local holomorphic coordinates, and so are ${(\xi_j^1(z,0), \ldots, \xi_j^n(z,0)}$. Hence:
        \[
        \det \left( \frac{\6 \xi_j^\alpha(z,0)}{\6z^\lambda} \right)_{j, \lambda} \neq 0 
        \]
        which implies that for small $|t|$:
        \[
        \det \left( \frac{\6 \xi_j^\alpha(z,t)}{\6z^\lambda} \right)_{j, \lambda} \neq 0 
        \]

        We denote by:
        \[
            A_{j,\alpha}^\lambda :=
            \left(\left(\frac{\6\xi_j^\alpha}{\6 z_\lambda} \right)_{j, \lambda} 
            \right)^{-1}
        \]
        and consider the local $1$-form
        \[
        \varphi_j^\lambda(z,t):=\sum_\alpha A^j_{j,\alpha} \bar\6 \xi_j^\alpha(z,t)
        \]

        \begin{claim}{\rm (\cite{kod})}
        $\varphi_j^\lambda$ is well-defined (i.e. does not depend on $\xi_j$) and is thus a global $T^{1,0}M$-valued $(0,1)$ form, denoted by:
        \[
        \varphi(t) \in A^{0,1}(T^{1,0}M)
        \]
        \end{claim}

        Thus, the explicit form of $\varphi(t)$ is
        \[
        \varphi(t) = 
        \left(\frac{\6}{\6z} \right)^T
        \left(\frac{\6 \xi}{\6 z}\right)^{-1}\bar\6 \xi
        \]

        Locally, $\varphi(t)$ can be described as:
        \[
        \varphi(t) = \varphi_{\bar j}^i d\bar z^j \otimes \frac{\6}{\6z^i}
        \]

        \begin{claim} {\rm (\cite{rz})}
        We have:
        \[
        \varphi_{\bar j}^i = \left(\left(\frac{\6 \xi}{\6 z}\right)^{-1}\left( \frac{\6 \xi}{\6z} \right)^T \right)^i_{\bar j}
        \]
        \end{claim}
    Also, one can see that $\varphi(0)=0$ and, in order for each $J_t$ to define an integrable complex structure on $M$, $\varphi(t)$ must satisfy the Maurer-Cartan equation, i.e.
    \begin{equation*}
        \bar \6 \varphi(t)=\frac{1}{2}[\varphi(t), \varphi(t)].
    \end{equation*}

    \begin{claim}
        Let 
        $\varphi = \varphi_{\bar j}^i d\bar z^j \otimes \frac{\6}{\6z^i}$
         and $ \bar \varphi = \bar \varphi_{ l}^{\bar k} d z^l \otimes \frac{\6}{\6 \bar z^k}$. We define
         $\varphi \bar \varphi:=\varphi_{\bar k}^i \bar \varphi_{\bar j}^k dz^j \otimes \frac{\6}{\6z^i}$.
        Then $ \bar \varphi  \llcorner \varphi=\varphi \bar \varphi$.
    \end{claim}

    \begin{proof}
        We have that $\varphi \in A^{0,1}(T^{1,0}M)$. Hence, $\bar \varphi \in A^{1,0}(T^{0,1}M)$. We compute the right-hand side and obtain
        \begin{align*}
            \iota_{\bar \varphi}\varphi=
            \bar \varphi  \intprod \varphi=
            \sum_{i,j,k,l} 
                 \left(dz^l \wedge \frac{\6}{\6 \bar z^k} d\bar z^j \otimes \frac{\6}{\6 z^i} \right) \bar \varphi_{\bar l}^k \varphi_{\bar j}^i 
                 &=
            \sum_{i,j,k,l} 
                 \left(dz^l \delta_j ^k \otimes \frac{\6}{\6 z^i} \right) \bar \varphi_{\bar l}^k \varphi_{\bar j}^i 
                &=
             \sum_{i,k,l} 
                    \bar \varphi_{\bar l}^k \varphi^{i}_{\bar k} dz^l \otimes \frac{\6}{\6 z^i} 
        \end{align*}.
    \end{proof}

    \begin{remark}
        We have $\varphi \in A^{0,1}(T^{1,0}M)$. Hence, $\varphi= \bar \eta \otimes Z$, where $\bar \eta \in A^{0,1}(M)$, $Z \in T^{1,0}(M)$.

        Then, 
        \begin{align*}
            \iota_\varphi dz^i=\bar \eta \wedge i_Z dz^i= Z(z^i)\bar \eta \in A^{0,1}(M).
        \end{align*}
    Moreover, 
    \begin{align*}
        \iota_\varphi^2 dz^i =\iota_\varphi(i_\varphi(dz^i))=
        \iota_\varphi(Z(z^i)\bar \eta)=
        Z(z_i)\bar \eta(Z) \bar \eta=0.
    \end{align*}

    Thus, $\iota_\varphi^2 dz^i= \cdots= \iota_\varphi^k dz^i=0$, for any $k \in \mathbb{N}$.
    
    Using the definition of $e^{\iota_{\varphi(t)}}$, we obtain $e^{\iota_{\varphi(t)}}(dz^i)=dz^i+\iota_\varphi dz_i$.

    \end{remark}

    \begin{proposition} {\rm (\cite[Lemma 2.5]{rz})}
    We have the following identities:
    \begin{enumerate}
        \item
        $d \xi^\alpha = \frac{\6 \xi^\alpha}{\6 z^i}
        \left( e^{\iota_{\varphi}}(dz^i) \right)$;
        
        \item 
        $\frac{\6 \xi^\alpha}{\6 z^i} \frac{\6}{\6 \xi^\alpha} =
        \left( (I - \varphi \bar \varphi)^{-1} \right)_i^j \frac{\6}{\6 z^j} -
        \left( (I - \bar \varphi \varphi)^{-1} \bar \varphi \right)_i^{\bar j} \frac{\6}{\6 \bar z^j}$.
    \end{enumerate}
\end{proposition}

    \begin{definition}
        Let ($\pi$, $\Mrond$, $I$) be a differentiable family of compact complex manifolds parametrized by $\varphi(t)$, for $I=(-\epsilon, \epsilon)$. For each $t$, we have
        \begin{equation*}
            d:A^{p,q}(M_t) \longrightarrow A^{p+1,q}(M_t) \bigoplus A^{p, q+1}(M_t).
        \end{equation*}
    Then, we define $\6_t:=\pi_t^{p+1,q} \circ d$ and
    $\bar \6_t:=\pi_t^{p,q+1} \circ d$.
   \end{definition}

\smallskip

    We need to recall the following fundamental proposition (see \cite {rz} \cite{kod}, \cite{nl}): 

    \begin{proposition} (\cite[Proposition 2.7]{rz})
        The holomorphic structure on $M_t$ is determined by $\varphi(t)$. Specifically, a differentiable function $f$ defined on any open subset of $M_0$ is holomorphic with respect to the holomorphic structure of $M_t$ if and only if
        \begin{equation*}
            \left (\bar \6 - \sum_i \varphi^i(t) \frac{\6}{\6 z^i} \right)
            f(z)=0,
        \end{equation*}
        where $\varphi^i(t)=\sum_j \varphi(t)_{\bar j}^i d \bar z^j$, or equivalently,
        
        \begin{equation*}
            (\bar \6- \varphi(t) \intprod \6) f(z)=0.
        \end{equation*}
    
     \end{proposition}

     \begin{remark}
    From the proof of the previous proposition, we can write the explicit form of $ \bar \6_t$:
    \begin{equation*}
        \bar \6_t f =
        e^{\iota_{\bar \varphi}}
        \left( (I - \bar \varphi \varphi)^{-1} \intprod
        (\bar \6 - \varphi \intprod \6) f
        \right)
    \end{equation*}
\end{remark}

    We need to understand the decomposition of each cotangent bundle $T^*M_t$ and we want to link $(p,q)$-forms form $M_0$ to $(p,q)$-forms from $M_t$. For this, we must use the next definition.

    \begin{definition}
        Let $\alpha \in A^{p,q}(M_0)$. We define the extension map
        \begin{align*}
             e^{\iota_{\varphi(t)}|\iota_{\bar \varphi(t)}}(\alpha)=\alpha_{i_1 \cdots i_p \bar j_1 \cdots  \bar j_q}(z)\left( e^{\iota_{\varphi(t)}}(dz^{i_1} \wedge \cdots dz^{i_p})\right)\wedge
            \left (e^{\iota_{\bar \varphi(t)}}(d \bar z^{j_1} \wedge \cdots  \wedge d \bar z^{j_q})
            \right ),
        \end{align*}
    where, locally, $\alpha=\alpha_{i_1 \cdots i_p \bar j_1 \cdots  \bar j_q} dz^{i_1} \wedge \cdots dz^{i_p} \wedge d \bar z^{j_1} \wedge \cdots  \wedge d \bar z^{j_q} $.

    \end{definition}
    \vspace{5 mm}

    We have the following result (see, \cite[Lemma 2.9]{rz}):

    \begin{proposition} 
    \label{ref:iso_p_q_forms_deformations}
        The extension map $e^{\iota_{\varphi(t)}|\iota_{\bar \varphi(t)}}: A^{p,q}(M_0) \longrightarrow A^{p,q}(M_t)$ is a linear isomorphism as $t$ is arbitrarly small. 
    \end{proposition}

    \begin{definition}
        We define the simultaneous contraction of a $(p,q)$ form by a $(0,1)$ vector form $\varphi$ as
        \begin{equation*}
            \varphi \cs \alpha:=
            \alpha_{i_1 \cdots i_p \bar j_1 \cdots \bar j_q }
            \varphi \intprod dz^{i_1}
            \wedge \cdots \wedge 
            \varphi \intprod dz^{i_p} \wedge 
            \varphi \intprod d \bar z^{j_1} \wedge 
            \cdots \varphi \intprod d \bar z^{j_q},
        \end{equation*}
    where, $\varphi \in A^{0,1}(T^{1,0}M_t)$ and $\alpha \in A^{p,q}M_t$ is written, locally, as $\alpha=\alpha_{i_1 \cdots i_p \bar j_1 \cdots \bar j_q } dz^{i_1} \wedge \cdots \wedge dz_{i_p} \wedge d \bar z^{j_1} \wedge \cdots  \wedge d \bar z^{j_q}$.
    \end{definition}

    We will prove a useful computational fact about the relation between the $e^{i_\varphi}$ operator and the wedge product.

    \begin{claim}\label{afirm:exp}
        \begin{equation*}
            e^{\iota_\varphi} dz^{i_1} \wedge \cdots \wedge e^{\iota_\varphi} dz^{i_p}=
            e^{\iota_\varphi}(dz^{i_1} \wedge \cdots \wedge dz^{i_p})
        \end{equation*}
    \end{claim}

    \begin{proof}
    Let $I':=I-\{i_1\}$. Then
    \begin{align*}
        \iota_\varphi(dz^{i_1} \wedge dz^{I'})=
        (\iota_\varphi dz^{i_1}) \wedge dz^{I'}+ dz^{i_1} \wedge \iota_\varphi(dz^{I'}).
    \end{align*}
    By induction, we can prove that 
    \begin{equation*}
        \iota_\varphi^k(dz^{i_1} \wedge dz^{I'})=
        k \iota_\varphi dz^{i_1} \wedge \iota_\varphi ^{k-1} dz^{I'}+
        dz^{i_1} \wedge \iota_\varphi^k(dz^{I'}),
    \end{equation*}
    for any $k$.

    Hence,
   \begin{equation*}
    \begin{split}
        e^{\iota_\varphi} & (dz^{i_1} \wedge dz^{I'}) = 
        \sum_{k=0}^{\infty}
        \frac{1}{k!} e_\varphi^k(dz^{i_1} \wedge dz^{I'}) =
        \sum_{k=0}^{\infty}
        \frac{1}{k!} 
        \left(k \, \iota_\varphi \, dz^{i_1} \wedge \iota_\varphi^{k-1} dz^{I'}
        + dz^{i_1} \wedge \iota_\varphi^k(dz^{I'})\right) \\
        & = \iota_\varphi dz^{i_1} \wedge 
        \left(
        \sum_{k=1}^\infty \frac{k}{k!} \iota_\varphi^{k-1} dz^{I'}
        \right) +
        dz^{i_1} \wedge
        \left(
        \sum_{k=0}^\infty \frac{1}{k!} \iota_\varphi^{k} (dz^{I'})
        \right)
    \end{split}
\end{equation*}
    Also, we have
    
    \begin{equation*}
        \begin{split}
            \sum_{k=1}^\infty \frac{k}{k!} \iota_\varphi^{k-1} dz^{I'}=
            \sum_{k=1}^\infty \frac{1}{(k-1)!} \iota_\varphi^{k-1} dz^{I'}=
            \sum_{k=0}^\infty \frac{1}{k!} \iota_\varphi^{k} dz^{I'}=e^{\iota_\varphi} dz^{I'}.
        \end{split}
    \end{equation*}

    Then,
    \begin{equation*}
        \begin{split}
            e^{\iota_\varphi} & (dz^{i_1} \wedge dz^{I'})=
            e^{\iota_\varphi} (dz^{i_1}) \wedge e^{\iota_\varphi} (dz^{I'})+
            d z^{i_1} \wedge e^{\iota_\varphi} (dz^{I'})
            = (e^{\iota_\varphi} (dz^{i_1})+ d z^{i_1}) \wedge e^{\iota_\varphi} (dz^{I'}) \\
            &=
            e^{i_\varphi} (dz^{i_1}) \wedge e^{i_\varphi} (dz^{I'})
        \end{split}
    \end{equation*}
    By repeating the process $n$ times, we obtain the conclusion.
        \end{proof}

\smallskip

    With the previous claim, we can write the extension map using simultaneous contraction.

    \begin{lemma} {\rm (\cite{rz})} \label{calcul}
    \begin{equation*}
        e^{\iota_\varphi|\iota_{\bar \varphi}}=(I+\varphi+ \bar \varphi) \cs 
    \end{equation*}
        
    \end{lemma}

    \begin{proof}
        Let $\alpha=\sum_{I,J} \alpha_{I,J} dz^I \wedge d \bar z^J$.

    Then,
  \begin{equation*}
    \begin{split}
        (I + \varphi + \bar \varphi) \cs \alpha &= \sum_{I,J} \alpha_{I,J} \left( (I + \varphi + \bar \varphi) \intprod d z^{i_1} \wedge \cdots \wedge (I + \varphi + \bar \varphi) \intprod d z^{i_p} \wedge (I + \varphi + \bar \varphi) \intprod d \bar z^{j_1} \wedge \cdots \right. \\
        & \left. \wedge (I + \varphi + \bar \varphi) \intprod d \bar z^{j_q} \right) = \sum_{I,J} \alpha_{I,J} (d z^{i_1}+ \iota_\varphi d z^{i_1}) \wedge \cdots \wedge (d z^{i_p}+ \iota_\varphi d z^{i_p}) \wedge (d \bar z^{j_1} + \iota_{\bar \varphi} d \bar z^{j_1}) \wedge \\
        & \cdots \wedge (d \bar z^{j_q} + \iota_{\bar \varphi} d \bar z^{j_q}) = \sum_{I,J} \alpha_{I,J} e^{\iota_\varphi}(d z^{i_1}) \wedge \cdots \wedge e^{\iota_\varphi}(d z^{i_p}) \wedge e^{\iota_{\bar \varphi}}(d \bar z^{j_1}) \wedge \cdots \wedge e^{\iota_{\bar \varphi}}(d \bar z^{j_q}).
    \end{split}
\end{equation*}

Using \ref{afirm:exp}, we obtain
\begin{equation*}
    \begin{split}
        \sum_{I,J} \alpha_{I,J} e^{\iota_\varphi}(d z^{i_1} \wedge \cdots \wedge d z^{i_p}) \wedge e^{\iota_{\bar \varphi}}(d \bar z^{j_1} \wedge \cdots \wedge d \bar z^{j_q}) &= e^{\iota_\varphi | \iota_{\bar \varphi}}(\alpha).
    \end{split}
\end{equation*}
\end{proof}

Thus, using the proof of \cite[Proposition 2.13]{rz}, we can write explicitly the action of $\6_t$ and $ \bar \6_t$ operator on $e^{i_\varphi|i_{\bar \varphi}}(\alpha)$, where $\alpha \in A^{p,q}(M_0)$. More precisely,

\begin{equation} \label{delt}
    \6_t(e^{\iota_{\varphi}|\iota_{\overline{\varphi}}}\alpha)=e^{\iota_{\varphi}|\iota_{\overline{\varphi}}}\left((I-\varphi\overline{\varphi})^{-1}\cs([ \bar \6,\iota_{\overline{\varphi}}]+\6)(I-\varphi\overline{\varphi})\cs\alpha\right)
\end{equation}
    
\begin{equation} \label{delbart}
    \bar \6_t(e^{\iota_{\varphi}|\iota_{\overline{\varphi}}}\alpha)=e^{\iota_{\varphi}|\iota_{\overline{\varphi}}}\left((I-\overline{\varphi}\varphi)^{-1}\cs([\6,\iota_{\varphi}]+\bar \6)(I-\overline{\varphi}\varphi)\cs\alpha\right).
\end{equation}

\subsection{Stability of special Hermitian metrics at deformations}\label{rez}

\medskip

We derive a necessary condition that the deformation must satisfy in order to preserve the property of a Hermitian metric to be special.

Let $\epsilon>0$ and $I:=(-\epsilon, \epsilon)$.
Suppose $\{M_t\}_{t\in I}$ is a small complex deformation (\ref{def:complex_deformation}) with $M_0=M$. As explained in Section \ref{sec:prelim}, the complex structures on $M_t$ can be described by a family of tensors $\varphi(t)\in A^{0,1}(T^{1,0}(M))$, for $t \in I$.
Let $\{\omega_t\}_{t\in I}$ be a smooth family of Hermitian metrics on $\{M_t\}_{t\in I}$. By \ref{ref:iso_p_q_forms_deformations}, for each $t \in I$ there exists $\omega(t) \in \mathcal{A}^{1,1}(M)$ such that:
\begin{equation*}
	\omega_t=e^{\iota_{\varphi}|\iota_{\overline{\varphi}}}\,\,(\omega(t)),
\end{equation*}
where $\omega(0) = \omega_0 = \omega$.

Locally, for each $k$ and index families $I = \{i_1, \ldots, i_k, j_1, \ldots, j_k | i_l < i_s, j_l < j_s, \forall l<s \}$, we find $\omega_I \in \mathcal{C}^\infty(I \times M)$ such that
\[ \omega^k(t) 
= 
\sum
\omega_I(t, \cdot)
dz^{i_1}\wedge d\overline{z}^{j_1}\wedge\dots\wedge dz^{i_{k}}\wedge d\overline{z}^{j_{k}}, \]
where the sum is taken over all index families as above.
We will also write $\omega_I(t) := \omega_I(t, \cdot)$.

Then $\omega_t^{k}$ has local expression 
$e^{\iota_{\varphi}|\iota_{\overline{\varphi}}}(\omega^{k}(t))
=e^{\iota_{\varphi}|\iota_{\overline{\varphi}}}
(\omega_I(t) dz^{i_1}\wedge d\overline{z}^{j_1}\wedge\dots\wedge dz^{i_{k}}\wedge d\overline{z}^{j_{k}})$, set
\[(\omega^{k}(t))'
:=
\frac{\6}{\6 t}(\omega_I(t))
dz^{i_1}\wedge d\overline{z}^{j_1}\wedge\dots dz^{i_{k}}\wedge d\overline{z}^{j_{k}}
\in \mathcal{A}^{k,k}(M).
\]

\begin{theorem}
	Let $\{ M_t \}_{t \in I}$ be a small complex deformation of a compact complex $n$-manifold and $\{ \omega_t \}_{t \in I}$ be a smooth family of Hermitian metrics along $\{ M_t \}_{t \in I}$. 
	If every metric $\omega_t$ satisfies $\6_t \bar \6_t \omega_t^k=0$, for $t\in I$, it must hold that
	\begin{equation}
		2i\mathfrak{Im}(\6 \circ \iota_{\varphi'(0)}\circ \6) (\omega^{k})=\6 \bar \6 (\omega^{k}(0))'.
	\end{equation}
\end{theorem}

\begin{proof}
	From hypothesis, we have that $\6_t \bar \6_t \omega_t^k=0$, for any $t \in I$. Then,
	\begin{equation*}
		\frac{\6}{\6 t}(\6_t \bar \6_t \omega_t^k)=0.
	\end{equation*}
	By definition of extension map, we have
	\begin{equation}
		\6_t \bar \6_t (\omega_t^k)= \6_t \bar \6_t(e^{i_\varphi|i_{\bar \varphi}}(\omega^k(t)))
	\end{equation}
	
	Using relations \ref{delt} and \ref{delbart}, the last relation became
	\begin{align*}
		&\6_t\bar \6_t(e^{\iota_{\varphi}|\iota_{\overline{\varphi}}}(\omega^{k}(t))=\6_t(e^{\iota_{\varphi}|\iota_{\overline{\varphi}}}\left((I-\overline{\varphi}\phi)^{-1}\cs([\del,\iota_{\varphi}]+\bar \6)(I-\overline{\varphi\varphi})\cs\,\omega^{k}(t)\right))\\
		&= e^{\iota_{\varphi}|\iota_{\overline{\varphi}}}\Biggl(\Bigl((I-\varphi\overline{\varphi})^{-1}\cs([\bar \6,\iota_{\overline{\varphi}}]+\6)(I-\varphi\overline{\varphi})\Bigr)\cs\Bigl((I-\overline{\varphi}\varphi)^{-1}\cs([\6,\iota_{\varphi}]+\bar \6)(I-\overline{\varphi}\varphi)\cs\,\omega^{k}(t)\Bigr) \Biggr).
	\end{align*}
	
	We will use the Taylor series expansion around $0$ for $\varphi$ and $\omega^k$ and we obtain
	\[
	\varphi(t)=t\varphi'(0)+o(t), \quad \omega^k (t)=\omega^{k}(0)+t\omega^{k}(0)'+o(t).
	\]
	
	So, 
	\begin{equation*}
		(I-\varphi\overline{\varphi})^{-1}=(I-\overline{\varphi}\varphi)^{-1}=(I-\varphi\overline{\varphi})=(I-\overline{\varphi}\varphi)=I+o(t)
	\end{equation*}
	
	Using \ref{calcul}, we have
	
	\begin{align*}
		&\6_t\bar \6_t\omega_t^{k}=\left(I+t\varphi'(0)+\overline{t\varphi}'(0)\right)\cs\,\left([\bar \6,\overline{t\varphi'(0)}\intprod]+\6\right)\left([\6,t\varphi'(0)\intprod]+\bar \6\right)\left(\omega^{k}(0)+t(\omega^{k}(0))'\right)+o(t).
	\end{align*}
	
	But we see that
	\begin{align*}
		\left([\6,t\varphi'(0)\intprod]+\bar \6\right)\left(\omega^{k}(0)+t(\omega^{k}(0))'\right)+o(t)=
		[\6,t\varphi'(0)\intprod] \omega^k(0)+ \bar \6 \omega^k(0)+t \bar \6(\omega^k(0)')+o(t).
	\end{align*}
	
	Then, 
	
	\begin{align} \label{calc_init}
		& \6_t \bar \6_t \omega_t^k = \left( I + t \varphi'(0) + \overline{t \varphi'(0)} \right) \cs \left( [\bar \6, \overline{t \varphi'(0)} \intprod ] + \6 \right) \left( [\6, t \varphi'(0) \intprod ] \omega^k(0) + \bar \6 \omega^k(0) + t \bar \6 (\omega^k(0)') \right)\nonumber \\
		& \quad + o(t).
	\end{align}
	
	We denote by $A:=[\6, t \varphi'(0) \intprod ] \omega^k(0) + \bar \6 \omega^k(0) + t \bar \6 (\omega^k(0)')$.
	
	Hence, 
	\begin{align*}
		\6& A= \6 \left(\6 (t \varphi(0)' \intprod \omega(0)^k)-t \varphi(0)' \intprod \omega^k(0) \right) + \6 \bar \6 \omega^k(0)+t \6\bar \6 (\omega^k(0)')=-t \6\left(\varphi(0)' \intprod \6 \omega^k(0) \right)+\\
		+t& \6 \bar \6 (\omega^k(0)').
	\end{align*}

	Thus, 
	\begin{align*}
		&[\bar \6,\overline{t\varphi'(0)}\intprod]+\6\left([\6, t \varphi(0)' \intprod] \omega^k(0)+\bar \6 \omega^k(0)+t \bar \6(\omega^k(0)')\right)+o(t)=[\bar \6,\overline{t\varphi'(0)}\intprod] \left([\6, t \varphi(0)' \intprod] \omega^k(0)\right)\\
		+& [\bar \6,\overline{t\varphi'(0)}\intprod](\bar \6 \omega^k(0))+ [\bar \6,\overline{t\varphi'(0)}\intprod](t \bar \6 \omega^k(0)')+ \6 A+ o(t).
	\end{align*}
	
	We observe that in the brackets of the first and third terms, there are terms containing t, so they will be absorbed into $o(t)$. The last equality became
	
	\begin{align*}
		& [\bar \6, \overline{t \varphi'(0)} \intprod](\bar \6 \omega^k(0)) + \6 A + o(t) = \bar \6 \left( \overline{t \varphi'(0)} \intprod \bar \6 \omega^k(0) \right) - t \varphi(0)' \intprod \bar \6 \bar \6 \omega^k(0) + \6 A + o(t) = \\ 
		& =t \bar \6 \left( \overline{\varphi'(0)} \intprod \6 \omega^k(0) \right) - t \6 \left( \varphi'(0) \intprod \omega^k(0) \right) + t \6 \bar \6(\omega^k(0)') + o(t).
	\end{align*}
	
	Hence, \ref{calc_init} becomes
	\begin{align*}
		\6_t \bar \6_t \omega_t^k = \left( I + t \varphi'(0) + \overline{t \varphi'(0)} \right) \cs \left(-t \6(\varphi'(0) \intprod \omega^k(0)) + t \bar \6 (\overline{\varphi'(0)} \intprod \bar \6 \omega^k(0))+ t \6 \bar\6 (\omega^k(0)') \right) +o(t).
	\end{align*}
	
	The relevant part is only the simultaneous contraction with the identity; the other two terms contain 
	$t$, and in the simultaneous contraction with forms that contain $t$, they will be absorbed in $o(t)$.
	
	Thus, 
	\begin{align*}
		\6_t \bar \6_t \omega_t^k =-t \6\left(\varphi'(0) \intprod \omega^k(0)\right) + t \bar \6 \left(\overline{\varphi'(0)} \intprod \bar \6 \omega^k(0) \right)+ t \6 \bar\6 (\omega^k(0)')  +o(t).
	\end{align*}
	
	From the hypotesis, $\frac{\6}{\6 t}\restrict{t=0}(\6_t\bar \6_t\omega_t^{k})=0$. Hence,
	\begin{equation*}
		-\6(\varphi'(0)\intprod \6\omega^{k}(0))+\bar \6(\overline{\varphi'(0)}\intprod\bar \6\omega^{k}(0))+\6\bar \6(\omega^{k}(0))'=0.
	\end{equation*}
	
	If we rewrite, we have
	\begin{equation*}
		-(\6\circ \iota_{\varphi'(0)}\circ\6) \omega^{k}+(\bar \6\circ \iota_{\overline{\varphi'(0)}}\circ\bar \6) \omega^{k}+\6\bar \6(\omega^{k}(0))'=0.
	\end{equation*}
	
	Hence,
	\begin{equation*}
		(-\6\circ \iota_{\varphi'(0)}\circ\6+\overline{ \6\circ \iota_{{\varphi'(0)}}\circ \6})\omega^k=-\6 \bar \6 (\omega^{k}(0))',
	\end{equation*}
	which is equivalent to 
	\begin{equation*}
		2i\mathfrak{Im}(\6 \circ \iota_{\phi'(0)}\circ \6) (\omega^{k})=\6 \bar \6 (\omega^{k}(0))'.
	\end{equation*}
\end{proof}

\section{Blow-up of special Hermitian metrics}
\subsection{Basic facts about blow-up}
\label{blow-up}

We will briefly present the construction of the blow-up of a complex manifold along a closed submanifold. This object will  also be a complex manifold together with a holomorphic map between it and the initial manifold. More details can be found in \cite{huy}.

Let $X$ be a complex manifold of dimension $n$ and $Y \subset X$ an arbitrary compact submanifold of dimension $m$. We choose an atlas $(U_i, \varphi_i)$, $\varphi_i: U_i \longrightarrow V_i:=\varphi_i(U_i)$ such that $\varphi_i(U_i \cap Y)=\varphi_i(U_i) \cap \mathbb{C}^m$.

Then $\varphi_j \circ (\varphi_i)^{-1}: V_i \longrightarrow V_j$ and
\begin{align*}
    (\varphi_j \circ (\varphi_i)^{-1})^k &= \sum_{s=m+1}^n x_s \varphi_{ks}^{ji}(x), \quad \text{for } k \geq m+1.
\end{align*}
We define 
\begin{equation*}
    Bl_{V_i \cap \mathbb{C}^m}(V_i):=\{(z,l)|z_i l_j=z_j l_i, i,j \geq m+1\} \subset V_i \times \mathbb{P}^{n-m-1}.
\end{equation*}

Let
\begin{align*}
    \Psi^{ij}: & \ Bl_{V_i \cap \varphi_i(U_j) \cap \mathbb{C}^m}(V_i \cap \varphi_i(U_j)) \longrightarrow Bl_{V_j \cap \varphi_j(U_i) \cap \mathbb{C}^m}(V_j \cap \varphi_j(U_i)) \\
    & (x,l) \mapsto \left( \varphi_j \circ \varphi_i^{-1}(x), \left( \varphi_{sr}^{ji} \right)_{s,r=m+1}^{n} \cdot l \right).
\end{align*}

Then $\Psi^{ij}$ is a well-defined biholomorphism and all the blow-ups on the various charts $\varphi_i(U_i)$ naturally glue, We have the following proposition:

\begin{proposition} {\rm (\cite{huy})}
    Let $Y$ be a complex submanifold of $X$. Then there exists a complex manifold $\hat{X}=Bl_Y(X)$, the blow-up of $X$ along $Y$, together with a  holomorphic map $\sigma: \hat{X} \longrightarrow X$ such that $\sigma: \hat{X} \setminus \sigma^{-1}(Y) \simeq X \setminus Y$ and $\sigma:\sigma^{-1}(Y) \longrightarrow Y$ is isomorphic with $\mathbb{P}(\mathcal{N}_{Y|X}) \longrightarrow Y$ 
\end{proposition}
\vspace{5 mm}

Let $(U_\alpha, \varphi_\alpha)$ be an atlas for $X$ with the previous properties, and let us denote $V_\alpha := \varphi_\alpha(U_\alpha)$ and $\Psi^\alpha: Bl_{V_\alpha \cap \mathbb{C}^m}(V_\alpha) \hookrightarrow{} Bl_Y(X)$. Let $pr_2: Bl_{V_\alpha \cap \mathbb{C}^m}(V_\alpha) \longrightarrow \mathbb{P}^{n-m-1}$ be the projection onto the second factor. We consider $\mathcal{O}_{\mathbb{P}^{n-m-1}}(-1)$ as the tautological line bundle, and we denote by $E_\alpha := pr_2^*\mathcal{O}_{\mathbb{P}^{n-m-1}}(-1)$. Let $h$ be the natural metric on $\mathcal{O}_{\mathbb{P}^{n-m-1}}(-1)$ (i.e., the metric such that, if we consider $e \in \mathcal{O}_{\mathbb{P}^{n-m-1}}(-1)$ and $pr(e) = [l]$, then there exists $\lambda \in \mathbb{C}$ such that $e = \lambda l$, where $l$ is a representative for $[l]$, and $h(e) = ||\lambda l||_{\mathbb{C}^{n-m-1}}$). We denote by $h' = pr_2^* h$ the metric on $E_\alpha$.

\begin{proposition}\cite[Proposition 3.25]{voi} \label{fibrat}
    There exists $E$ a vector bundle over $Bl_Y(X)$ such that $(\Psi^\alpha)^*E=E_\alpha$. Moreover:
    \begin{enumerate}
        \item \label{fibrat trivial peste blowup} $E$ is trivial over $Bl_Y(X)\setminus \sigma^{-1}(Y)$
        \item If we make the identification $\sigma^{-1}(Y) \simeq \mathbb{P}(\mathcal{N}_{Y|X})$, we have $E|_{\sigma^{-1}(Y)} \simeq \sigma^*\mathcal{N}_{Y|X}$
    \end{enumerate}
        
\end{proposition}

\begin{remark}
    \ref{fibrat}, \eqref{fibrat trivial peste blowup} implies that there exists $\eta$ a nowhere vanishing global section. We can choose $h''$ a metric on $E|_{Bl_Y(X)\setminus \sigma^{-1}(Y)}$ such that $h''(\eta)=1$. As $\sigma^{-1}(Y)$ is compact, then there exists $W_1$, $W_2$ relatively compact open sets, with $\overline W_1 \subset W_2$, such that $\sigma^{-1}(Y) \subset W_1$. We choose $\rho_1$, $\rho_2$ $\mathcal{C}^\infty$ functions such that $\rho_1=1$ on $W_1$ and $\rho_1=0$ outside $W_2$ and $\rho_2=0$ on $W_1$ and $\rho_2=1$ outside $W_2$.
\end{remark}

We define $H:=\rho_1 h'+ \rho_2 h''$ metric on $E$.

\begin{proposition} {\rm(\cite{gh})} \label{curbura}
    The curvature $\omega$ of the canonical connection associated with metric $H$ on $E$ has the following properties:
    \begin{enumerate}
        \item  $\omega$ is zero outside $W_2$;
        \item   $- \omega$ is semi-negative definite on $W_1$
        \item the restriction of it at $\sigma^{-1}(Y)$ is negative definite on vectors that are tangent to the  fibre of the bundle $\sigma^* \mathcal{N}_{Y|X} \longrightarrow \mathbb{P}(\mathcal{N}_{Y|X})$
    \end{enumerate}
\end{proposition}

\medskip

\subsection{Stability of special Hermitian metrics at blow-up}

Here we investigate the stability of a special Hermitian metric with property that its fundamental $2-$form satisfies $\6 \bar \6 F^i=0$ with respect to blow-up, generalizing the result for SKT metrics  in \cite{ft}.

\begin{theorem}
    Let $X$ be a complex manifold of dimension $n$ endowed with a Hermitian metric $g$ such that, for a fixed $k>0$, the fundamental $2$-form $F$ of $g$ satisfies $\6 \bar \6 F^i=0$, for any $i=\overline{1,k}$.

    Let $Y \subset X$ be a compact complex submanifold.

    Then the blow-up $Bl_Y(X)$ of $X$ along $Y$ admits a Hermitian metric such that its fundamental $2$-form $\tilde{F}$ satisfies $\6 \bar \6 \tilde{F}^i=0$, for any $i=\overline{1,k}$ .
\end{theorem}

\begin{proof}
    Let $\sigma:Bl_Y(X) \longrightarrow X$. We know that $\sigma:Bl_Y(X) \setminus \sigma^{-1}(Y)\longrightarrow X \setminus Y$ is biholomorphism and $\sigma^{-1}(Y) \simeq \mathbb{P}(\mathcal{N}_{Y|X})$.

    From \ref{fibrat}, we know that there exists a holomorphic line bundle $E$ over $Bl_Y(X)$ which admits a Hermitian metric. This metric has the properties listed in  \ref{curbura} .

    We denote by $\omega$ the curvature of the canonical connection associated to this metric.
    Then $\omega$ is zero outside a compact neighborhood of $\sigma^{-1}(Y)$, $-\omega$ is semi-negative definite on any vector and $-\omega|_{\sigma^{-1}(Y)}$ is negative definite on vectors that are tangent to any fibre of $\sigma^* \mathcal{N}_{Y|X} \longrightarrow \mathbb{P}(\mathcal{N}_{Y|X})$

    As $Y$ is compact, there exists $N>>0$ such that $\tilde{F}= \sigma^*F+ N\omega$ is positive definite.

    Then,
    \begin{equation*}
        \tilde{F}^k=(\sigma^*F+N\omega)^k=\sum_{l=0}^k C_l^k (\sigma^*F)^l \wedge  (N\omega)^{k-l}.
    \end{equation*}

    Hence, 
    \begin{align*}
        \6  \bar \6 \tilde{F}^k&= \sum_{l=0}^k C_l^k \left[ \6 \bar \6 (\sigma^*F)^l \wedge (N \omega)^{k-l}- \bar \6(\sigma^*F)^l \wedge \6  (N \omega)^{k-l}\right.\\
        &+  \left.\6(\sigma^*F)^l \wedge \bar \6  (N \omega)^{k-l} + \6 (\sigma^*F)^l \wedge \6 \bar \6 (N \omega)^{k-l}\right].
    \end{align*}

    Since $\omega$ is the curvature form of a connection, it is $d$-closed. Then $\6 \omega=\bar \6 \omega=0$. Thus, the last $3$ terms of the sum reduce to zero.

    In conclusion, 
    \begin{align*}
        \6  \bar \6 \tilde{F}^k&=\sum_{l=0}^k C_l^k \6 \bar \6 (\sigma^*F)^l \wedge (N \omega)^{k-l}=\sum_{l=0}^k C_l^k N^{k-l} \6 \bar \6 (\sigma^*F)^l \wedge \omega^{k-l}\\
        &=\sum_{l=0}^k C_l^k N^{k-l} \6 \bar \6 \sigma^*F^l \wedge \omega^{k-l}
        =\sum_{l=0}^k C_l^k N^{k-l}  \sigma^* \6 \bar \6 F^l \wedge \omega^{k-l}.
    \end{align*}

    From the hypothesis, we have $\6 \bar \6 F^l=0$. Hence,  the conclusion. 
\end{proof}

\begin{corollary}
	Let $X$ be a complex manifold endowed with a special Hermitian metric. Then the blow-up along a submanifold of $X$ admits a special Hermitian metric.
\end{corollary}

\begin{remark}
The above computation for $F^k$ can be applied to any sum of two Hermitian metrics and shows that the product between a  K\"ahler manifold and a special Hermitian  manifold is special Hermitian. In particular, the problem being local, we derive that any principal torus bundle over a special Hermitian manifold is special Hermitian. This provides new examples out of Endo-Pajitnov manifolds.
\end{remark}

\hfill

\noindent{\bf Acknowledgment:} My warm thanks to Alexandra Otiman for suggesting this research topic and to Liviu Ornea and Miron Stanciu for many useful discussions.


\begin{thebibliography}{100}
	
		\bibitem[AI01]{_Alexandrov_Ivanov_}
		Alexandrov, Bogdan; Ivanov, Stefan,
		{\em Vanishing theorems on Hermitian manifolds,}
		Differential Geom. Appl. 14 (2001), no. 3, 251-265.

         \bibitem[B87]{b87} A. Besse, {\it Einstein Manifolds}, Springer Verlag, 1987
        
        \bibitem[COS]{com}  C. Ciulică, A. Otiman, M. Stanciu {\em Special non-Kähler metrics on Endo-Pajitnov manifolds}, arXiv:2403.15618 (2024).
        
        \bibitem[EP20]{pajitnov1} H. Endo, A. Pajitnov, {\em On generalized Inoue manifolds}, Proceedings of the International Geometry Center \textbf{13} 4 (2020), 24--39.
        
        
        
        \bibitem[FGV19]{fgv19} A. Fino, G. Grantcharov, L. Vezzoni, {\it Astheno-\K \ and Balanced Structures on Fibrations}, Int. Math. Res. \textbf{2019} 22, 7093-7117
        
        
        \bibitem[FT09]{ft} A. Fino, A. Tomassini, {\it Blow-ups and resolutions of strong Kähler with torsion metrics}, Advances in Mathematics \textbf{221} (2009), 914--935.
        
        \bibitem[FV16]{fv} A. Fino, L. Vezzoni, {\it On the existence of balanced and SKT metrics on nilmanifolds}, Proc. Amer. Math. Soc., 144(6), 2016, 2455–2459.
        
        \bibitem[FWW13]{fww13} J. Fu, Z. Wang, D. Wu, {\it Semilinear equations, the $\gamma_k$ function, and generalized Gauduchon metrics}, J. Eur. Math. Soc. \textbf{15} (2013), 659--680.
        
        \bibitem[GH94]{gh} P. Griffiths, J. Harris {\em Principles of Algebraic Geometry}, Wiley-Interscience (1994).
        
        \bibitem[H04]{huy} D. Huybrechts {\em Complex Geometry}, Springer Berlin, Heidelberg (2004).
        
        \bibitem[JY93]{jy} J. Jost, S.-T. Yau, {\it A nonlinear elliptic system for maps from Hermitian to Riemannian manifolds and rigidity theorems in Hermitian geometry}, Acta Math. \textbf{170} 2 (1993), 221--254. 
        
        \bibitem[K04]{kod_2} K. Kodaira {\em Complex Manifolds and Deformation of Complex Structures}, Springer Berlin, Heidelberg (2004).
        
        
        \bibitem[LRY15]{lry}  K. Liu, S. Rao, X. Yang, {\it Quasi-isometry and deformations of Calabi-Yau manifolds}, Invent. Math. \textbf{199} (2015), 423--453.
        
        \bibitem[M82]{m82} M.L. Michelson, {\it On the existence of special metrics in complex geometry}, Acta Math. \textbf{149} (1982), 261--295.
        
        \bibitem[MK71]{kod} J. Morrow, K. Kodaira {\em Complex Manifolds}, AMS Chelsea Publishing (1971).
        
        \bibitem[NL57]{nl} A. Newlander, L. Nirenberg, {\it Complex analytic coordinates in almost complex manifolds}, Ann. of Math. \textbf{65} (1957), 391--404.
        
        \bibitem[OV24]{ov24} L. Ornea, M. Verbitsky, {\em Principles of Locally Conformally K\" ahler Geometry}, Birkhäuser Cham (2024).
        
        \bibitem[OVV13]{ovv13} L. Ornea, M. Verbitsky, V.Vuletescu {\em Blow-ups of Locally Conformally Kähler Manifolds}, International Mathematics Research Notices\textbf{12} (2013), 2809--2821.
        
        \bibitem[PS21]{ps} R. Piovani, T. Sferruzza, {\it Deformations of strong K\"ahler with torsion metrics}, Complex Manifolds \textbf{8} (2021), 286--301.
        
        
        \bibitem[RZ18]{rz} S. Rao, Q. Zhao, {\it Several special complex structures and their deformation properties}, J. Geom. Anal. \textbf{28} (2018), 2984--3047.
        
        
        \bibitem[S23]{sfe} T. Sferruzza, {\it Deformations of astheno-Kähler metrics}, Math. Ann. \textbf{10} 1 (2023), 20230102.
        
        \bibitem[S22]{sfe_2} T. Sferruzza, {\it Deformations of balanced metrics}, Bul. Sci. Math \textbf{178} (2022), 103143.
        
        \bibitem[ST23]{sfto} T. Sferruzza, A. Tomassini, {\em On cohomological and formal properties of strong K\"hler with torsion and astheno-K\"hler metrics}, Math. Z. {\bf 304} (2023) art. no. 55.
        
        \bibitem[Vo10]{voi} C. Voisin {\em Hodge Theory and Complex Algebraic Geometry I}, Cambridge University Press (2010).
        
        \bibitem[Vu09]{vul} V. Vuletescu, {\it Blowing-up points on l.c.K. manifolds}, Bull. Math. Soc. Sci. Math. Roumanie(N.S.) \textbf{52(100)} (2009), no.3, 387--390.
        
        
        
\end{thebibliography}
\end{document}